\documentclass[11pt]{amsart}
\usepackage[ruled,vlined,linesnumbered,norelsize]{algorithm2e}
\usepackage{amsmath}
\usepackage{amssymb}
\usepackage{booktabs}
\usepackage{url}

\newtheorem{thm}{Theorem}
\newtheorem{prop}[thm]{Proposition}

\theoremstyle{remark}

\theoremstyle{definition}

\newcommand{\ZZ}{\mathbf{Z}}
\newcommand{\CC}{\mathbf{C}}
\newcommand{\RR}{\mathbf{R}}
\newcommand{\QQ}{\mathbf{Q}}
\newcommand{\FF}{\mathbf{F}}
\newcommand{\assign}{\leftarrow}
\newcommand{\divides}{\mathrel{|}}
\newcommand{\ndivides}{\mathrel{\nmid}}
\newcommand{\ceil}[1]{\lceil {#1} \rceil}
\newcommand{\floor}[1]{\lfloor {#1} \rfloor}
\newcommand{\eps}{\varepsilon}

\DeclareMathOperator{\LCM}{LCM}

\begin{document}

\title{A search for Wilson primes}

\author{Edgar Costa}
\address[Edgar Costa]{Courant Institute of Mathematical Sciences \\
New York University \\
251 Mercer Street \\
New York, N.Y. 10012-1185 \\
U.S.A}
\thanks{The first author was partially supported by FCT doctoral grant SFRH/BD/ 69914/2010.}
\email{edgarcosta@nyu.edu}

\author{Robert Gerbicz}
\address[Robert Gerbicz]{E\"otv\"os Lor\'and University \\
H-1117 Budapest, P\'azm\'any P\'eter s\'et\'any 1/C, Hungary}
\email{robert.gerbicz@gmail.com}

\author{David Harvey}
\address[David Harvey]{School of Mathematics and Statistics \\
University of New South Wales \\
Sydney NSW 2052 \\
Australia}
\thanks{The third author was partially supported by the Australian Research Council, DECRA Grant DE120101293.}
\email{d.harvey@unsw.edu.au}

\begin{abstract}
A Wilson prime is a prime $p$ such that $(p-1)! = -1 \pmod{p^2}$. We report on a search for Wilson primes up to $2 \times 10^{13}$, and describe several new algorithms that were used in the search. In particular we give the first known algorithm that computes $(p-1)! \pmod{p^2}$ in average polynomial time per prime.
\end{abstract}

\maketitle

\section{Introduction}

Wilson's theorem in elementary number theory states that
 \[ (p-1)! = -1 \pmod p \]
for any prime $p$. The corresponding \emph{Wilson quotient} is
 \[ \frac{(p-1)! + 1}p \in \ZZ, \]
and we define $w_p$ to be its residue modulo $p$ in the interval $-p/2 \leq w_p < p/2$. A \emph{Wilson prime} is a prime such that $w_p = 0$, or equivalently
 \[ (p-1)! = -1 \pmod{p^2}. \]
Only three Wilson primes are known: $5$, $13$ and $563$.

All previously published searches for Wilson primes have used algorithms for computing $w_p$ whose time complexity is essentially linear in $p$. (In this paper, unless otherwise specified, time complexity means number of steps on a multitape Turing machine, see \cite{Pap-complexity}.) Since the input size is $\Theta(\log p)$, these algorithms should be regarded as having exponential time complexity. For example, the simplest possible algorithm is to multiply successively by the integers $2, 3, \ldots, p - 1$, reducing modulo $p^2$ after each multiplication. The best known algorithm for computing $w_p$ has complexity $p^{1/2 + \eps}$ (see below), but this is still exponential in $\log p$. Here and below, $X^\eps$ means $X^{o(1)}$, where $o(1)$ is a quantity approaching zero as $X \to \infty$.

The main theoretical contribution of this paper is an algorithm that computes $w_p$ in \emph{polynomial time on average}:
\begin{thm}
\label{thm:main}
The Wilson quotients $w_p$ for $2 \leq p \leq N$ may be computed in time $N \log^{3+\eps} N$.
\end{thm}
Let $\pi(x)$ denote the number of primes $p \leq x$. By the prime number theorem, $\pi(x) \sim x/\log x$, so Theorem \ref{thm:main} implies that we can compute each $w_p$ in time essentially $\log^4 p$ on average. While this result does not improve the complexity for computing a single $w_p$, it is of course directly relevant to the problem of searching for Wilson primes. 

The key idea of the new algorithm is to exploit redundancies among the products $(p-1)!$ for varying $p$. For example, the Wilson quotients for $N < p < 2N$ in some sense all incorporate the product $N!$. Instead of computing $N! \pmod{p^2}$ separately for each $p$, we will compute it modulo the product $\prod_{N < p < 2N} p^2$. A remainder tree then yields $N! \pmod{p^2}$ for each $p$. Using FFT methods for integer arithmetic, this can all be achieved in average polynomial time per prime. Applying this idea recursively leads to an algorithm for computing the desired residues $(p-1)! \pmod{p^2}$. A detailed description is given in the proof of Theorem \ref{thm:main} in Section \ref{sec:average}.

However, the space requirements of this algorithm render it impractical for large $N$, and we must implement a time-space tradeoff to obtain a practical algorithm:
\begin{thm}
\label{thm:main-space}
Let $M < N$, and assume that $N - M > \sqrt N \log N \log \log N$. The Wilson quotients $w_p$ for $M < p \leq N$ may be computed in time
 \[ M \log^{2+\eps} M + (N - M + \sqrt N) \log^{3+\eps} N \]
and space $O(N - M)$.
\end{thm}

The algorithm implementing Theorem \ref{thm:main-space} consists of two main phases that we call Stage 1 and Stage 2. Stage 1 involves computing $M!$ modulo $\prod_{M < p \leq N} p^2$, and contributes the $M \log^{2+\eps} M$ term to the time bound. Stage 2, which contributes the second term, is a modification of the algorithm implementing Theorem \ref{thm:main}.

The average time per prime in Stage 2 is essentially $\log^4 p$, the same as for Theorem \ref{thm:main}. However in Stage 1 the average time per prime behaves like
 \[ \frac{p}{N - M} \log^3 p. \]
This is no longer polynomial in $\log p$, and represents the price we pay for restricting the space consumption. If we now assume that the amount of RAM is fixed, then a reasonable strategy to compute $w_p$ for all $p$ up to some bound $N_0$ is to apply Theorem \ref{thm:main-space} to successive intervals $M < p \leq N$, where $N \leq N_0$, and where $N - M$ is chosen as large as possible given the available RAM.

This is in fact what we did, for all $p < 2 \times 10^{13}$. We found no new Wilson primes up to this bound. Altogether this consumed over 1.1 million hours of CPU time. It is traditional, though meaningless, to give tables of `near misses'. Table \ref{tab:small1} shows the smallest $|w_p|$ that we found, and Table \ref{tab:small2} shows the smallest residues when ordered by $|w_p/p|$.

\begin{table}
\begin{tabular}{rrcrr}
\toprule
$p$ & $w_p$ & \phantom{------} & $p$ & $w_p$ \\
\midrule
          56\,151\,923 & $-1$  & &       4\,036\,677\,373 & $-5$ \\
     11\,774\,118\,061 & $-1$  & &  5\,609\,877\,309\,359 & $-6$ \\
          14\,296\,621 & $+2$  & &           10\,746\,881 & $-7$ \\
          87\,467\,099 & $-2$  & &           11\,892\,977 & $-7$ \\
16\,556\,218\,163\,369 & $+2$  & &           39\,198\,017 & $-7$ \\
           8\,315\,831 & $+3$  & &       1\,767\,839\,071 & $+8$ \\ 
          93\,559\,087 & $-3$  & &           29\,085\,907 & $+9$ \\
          51\,802\,061 & $+4$  & &      67\,133\,912\,011 & $+9$ \\
    258\,818\,504\,023 & $+4$  & &      42\,647\,052\,491 & $+10$ \\
 1\,239\,053\,554\,603 & $-4$  & &     935\,606\,702\,249 & $-10$ \\
      1\,987\,272\,877 & $+5$  \\
\bottomrule
\end{tabular}
\caption{Primes $10^6 < p < 2 \times 10^{13}$ for which $|w_p| \leq 10$}
\label{tab:small1}
\end{table}

\begin{table}
\begin{tabular}{rrcrr}
\toprule
$p$ & $w_p$ & \phantom{------} & $p$ & $w_p$ \\
\midrule
                      5 &              $0$ & & 17\,475\,368\,544\,847 &           $+154$ \\
                     13 &              $0$ & & 13\,561\,740\,531\,809 &           $+120$ \\
                    563 &              $0$ & &  9\,461\,354\,987\,597 &            $+94$ \\
 16\,556\,218\,163\,369 &             $+2$ & & 13\,707\,091\,918\,909 &           $+143$ \\
  5\,609\,877\,309\,359 &             $-6$ & &     935\,606\,702\,249 &            $-10$ \\
 14\,875\,476\,519\,749 &            $-38$ & &  1\,108\,967\,825\,921 &            $+12$ \\
 15\,395\,725\,531\,427 &            $+46$ & &  2\,170\,161\,095\,393 &            $+25$ \\
  1\,239\,053\,554\,603 &             $-4$ & & 16\,690\,620\,863\,071 &           $+203$ \\
  4\,663\,421\,363\,459 &            $+28$ & &  2\,462\,223\,083\,147 &            $-35$ \\
  7\,746\,014\,299\,613 &            $+47$ & & 17\,524\,177\,394\,617 &           $+256$ \\
 11\,273\,815\,078\,217 &            $+88$ & & 10\,865\,903\,332\,033 &           $+159$ \\
  7\,338\,481\,259\,891 &            $-62$ & & 16\,880\,979\,600\,449 &           $+253$ \\
\bottomrule
\end{tabular}
\caption{Primes $p < 2 \times 10^{13}$ for which $|w_p/p| \leq 1.5 \times 10^{-11}$}
\label{tab:small2}
\end{table}

Retaining all of the residues would have required archival storage in the terabyte range. Instead, we only recorded those residues for which $|w_p| \leq p/50000$, i.e.~approximately 0.004\% of the primes examined. There are $27\,039\,026$ such primes; the residues may be downloaded from the third author's web page (247 MB).

The search for Wilson primes has an interesting history. The case $p = 5$ is trivial, and $p = 13$ was noticed at least as early as 1892 \cite[p.~318]{Mat-theory}. In 1913, Beeger used the congruence
 \[ w_p = B_{p-1} - \frac{p-1}p \pmod p, \]
where $B_k$ is the $k$-th Bernoulli number, together with a published table of Bernoulli numbers, to check that there are no other Wilson primes less than $114$ \cite{Bee-congruence}. Several years later he proved the congruence
\begin{equation}
\label{eq:reduce-2}
 (p-1)! = (-1)^{(p-1)/2} \left(\left(\frac{p-1}2\right)!\right)^2 (2^p - 1) \pmod{p^2},
\end{equation}
which reduces computation of $w_p$ to that of $((p-1)/2)! \pmod{p^2}$. He used this identity, together with a direct computation of the relevant factorials, to produce a table of $w_p$ for $p < 300$ \cite{Bee-congruence2}. We do not know when \eqref{eq:reduce-2} was first discovered, but it appears (without proof) in \cite{Mat-theory}. 

Lehmer later used Beeger's original method together with a newly extended table of Bernoulli numbers to compute $w_p$ for $p \leq 211$ \cite{Leh-wilson}. In a companion article, she mentions that Beeger communicated that his earlier table contains four errors, namely for $p = 127$, $167$, $173$ and $241$ \cite{Leh-congruence}. Lehmer's table is correct, but there is an additional unnoticed error in Beeger's table, for $p = 239$. The errors are rather clustered together, and one speculates on the human factors (computational exhaustion?) that may have been responsible. For the modern reader, it is very easy to forget just how much effort is required to generate such a table by hand. We invite the reader to spend a few minutes verifying that $p = 13$ is indeed a Wilson prime!

After these early attempts, the search entered the computer age with the work of Goldberg, who used the Bureau of Standards Eastern Automatic Computer (SEAC), one of the first stored-program electronic computers, to test all $p < 10\,000$ \cite{Gol-wilson}. In this interval, not far beyond the previous search bound, was found the third Wilson prime $p = 563$. Fr\"oberg pushed this further to $30\,000$ and then $50\,000$ \cite{Fro-wilson-fermat, Fro-wilson}. In \cite{Fro-wilson} he also discusses a heuristic concerning the distribution of Wilson primes. Namely, if one assumes that $w_p$ is uniformly distributed modulo $p$, then the probability that $p$ is a Wilson prime is $1/p$, and the expected number of Wilson primes less than $X$ is
 \[ \sum_{p < X} \frac1p = \log \log X + c + o(1), \]
where $c = 0.2615...$ is Mertens' constant. This suggests that there should be infinitely many Wilson primes, but that they should be very rare.

The search bound was successively increased to $200\,183$ by Pearson \cite{Pea-wilson-fermat}, $1\,017\,000$ by Kloss \cite{Klo-number-theoretic}, $3\,000\,000$ by Keller (see \cite[p.~350]{Rib-new-records}), $4\,000\,000$ by Dubner \cite{Dub-wilson}, $10\,000\,000$ and then $18\,876\,041$ by Gonter and Kundert \cite{Kun-von-staudt}. (The computation was halted at $18\,876\,041$ due to a power failure --- see \cite[p.~350]{Rib-new-records}. Many authors have cited an unpublished manuscript ``All prime numbers up to 18,876,041 have been tested without finding a new Wilson prime'' by Gonter and Kundert, but we have been unable to locate a copy.)

None of these authors give many details on how they performed the computation. It seems likely that they were all aware of \eqref{eq:reduce-2}, and that they computed $((p-1)/2)! \pmod{p^2}$ by simply multiplying successively by $2$, $3$, \ldots, $(p-1)/2$, reducing modulo $p^2$ at frequent intervals.

Significant algorithmic progress on the problem was made by Crandall, Dilcher and Pomerance, who searched up to $5 \times 10^8$ \cite{CDP-search}. They introduced two new main ideas. The first is that for many $p$, there exist identities better than \eqref{eq:reduce-2}. For example, if $p = 1 \pmod 4$, write $p = a^2 + b^2$ with $a = 1 \pmod 4$. Then we have the remarkable identity (proved in \cite{CDE-binom})
 \[ \binom{\frac12(p-1)}{\frac14(p-1)} = \left(1 + \frac{2^{p-1} - 1}2\right)\left(2a - \frac{p}{2a}\right) \pmod{p^2}. \]
Together with \eqref{eq:reduce-2} this reduces the computation of $w_p$ to that of $((p-1)/4)! \pmod{p^2}$. Similar identities are used in \cite{CDP-search} to reduce to computation of $((p-1)/6)! \pmod{p^2}$ in the case that $p = 1 \pmod 6$.

We extend this technique considerably in Section \ref{sec:identities}, showing how to reduce to computation of $((p-1)/e)! \pmod{p^2}$ for essentially any `small' divisor $e$ of $p-1$.

Second, \cite{CDP-search} introduced a scheme that replaces most of the modular multiplications by modular additions. Indeed they show how to compute $N! \pmod{p^2}$ using $N + O(N^{2/3})$ additions and only $O(N^{2/3})$ multiplications. This optimisation does not play a role in the present work.

Crandall--Dilcher--Pomerance also mention an algorithm, essentially due to Strassen, that computes $(p-1)! \pmod{p^2}$ in time $p^{1/2+\eps}$; however they found it was not competitive with their quasi-linear time algorithm over the range of their search. This can be improved by a factor of $\log p$ \cite{BGS-recurrences}, yielding the best known algorithm for computing a single $w_p$.

Following this work, Carlisle--Crandall--Rodenkirch extended the search to $10^9$ in 2006 (see \cite[p.~241]{RK-primzahlen}) and then $6 \times 10^9$ in 2008 (personal communication). This work has not been published; we sketch their algorithm here. The basic idea is to explicitly compute the exponents appearing in the prime factorisation $N! = p_1^{e_1} \cdots p_r^{e_r}$, and then compute this product, term by term, modulo $p^2$. The complexity is $O(N / \log N)$ multiplications, which improves on the algorithms used in \cite{CDP-search} by a factor of $\log N$.

\section{Computing Wilson quotients in average polynomial time}
\label{sec:average}

In this section we give algorithms that prove Theorems \ref{thm:main} and \ref{thm:main-space}. The algorithms depend on three fundamental operations: integer multiplication, integer division, and enumeration of primes. We discuss the complexity of these operations first. We will give only a high level description of all algorithms, allowing the industrious reader to supply their own details concerning data layout and access patterns by the Turing machine.

If $X$ and $Y$ are integers with at most $N$ bits, their product can be computed in time $N \log^{1+\eps} N$ and space $O(N)$ using FFT methods \cite{SS-multiply, Fur-faster}. For division with remainder, we want $Q = \floor{X/Y}$ (assuming $Y > 0$) and $R = X \bmod Y$. These can also be computed in time $N \log^{1+\eps} N$ and space $O(N)$ \cite{Ber-fastmult}.

Consider the problem of enumerating the primes $M < p \leq N$. In our implementation (see section \ref{sec:implementation}) we used a simple sieve of Eratosthenes, i.e.~after precomputing a table of primes $q \leq \sqrt{N}$, we initialise a bit-array of length $N - M$ and strike out multiples of each $q$ to eliminate the composites. Assuming a RAM model with unit time access to arbitrary array elements, and in which integers of size $O(\log N)$ can be manipulated in unit time, the complexity is at most
 \[ \sum_{\substack{q \leq \sqrt N \\ \text{$q$ prime}}} \left\lceil \frac{N - M}q \right\rceil \leq \sum_{\substack{q \leq \sqrt N \\ \text{$q$ prime}}} \left(\frac{N - M}q + 1\right) = O((N - M) \log \log N + \sqrt N) \]
by Mertens' theorem.

While this simple algorithm is perfectly adequate in practice, in the Turing model the analysis is incorrect, because of the unavailability of unit-time array access. For completeness, Proposition \ref{prop:enumerate-interval} below gives a bound for the Turing model, following the approach suggested in \cite[p.~226]{SGV-fast-algorithms}. This result is not optimal, but suffices for our purposes. The key tool is merge sort, which can be implemented efficiently on a Turing machine; see \cite{Die-turing} for a discussion of this, and for further applications of this observation in computational number theory.

\begin{prop}
\label{prop:enumerate}
The primes $p \leq N$ may be enumerated in time
 \[ O(N \log^2 N \log \log N) \]
and space
 \[ O(N \log N \log \log N). \]
\end{prop}
\begin{proof}
First enumerate the primes $q \leq \sqrt N$ by trial division. There are $O(\sqrt N)$ candidates, and each requires $O(N^{1/4})$ divisibility tests, so the time cost is $N^{3/4 + \eps}$.

Now for each $q \leq \sqrt N$, generate the multiples of $q$ bounded by $N$. The number of such multiples is $d = \sum_{q \leq \sqrt N} \floor{N/q} = O(N \log \log N)$. Each successive multiple is computed via a single addition of integers of size $O(\log N)$, so the time and space required to construct the list is $O(d \log N)$. Sort the list using merge sort; this costs time $O(d \log d \log N) = O(N \log^2 N \log \log N)$ and space $O(d \log N) = O(N \log N \log \log N)$. The complement of the resulting list in $1 \leq x \leq N$ is the desired set of primes, and can be computed in one more pass in time $O(d \log N)$.
\end{proof}

\begin{prop}
\label{prop:enumerate-interval}
The primes $M < p \leq N$ may be enumerated in time
 \[ O((N - M + \sqrt N) \log^2 N \log \log N) \]
and space
 \[ O((N - M + \sqrt N) \log N \log \log N) \]
\end{prop}
\begin{proof}
First enumerate the primes $q \leq \sqrt N$ using Proposition \ref{prop:enumerate}. This requires time $O(\sqrt N \log^2 N \log \log N)$ and space $O(\sqrt N \log N \log \log N)$.

Now for each $q \leq \sqrt N$, generate the multiples of $q$ in the interval $M < x \leq N$. Determining the first multiple of each $q$, namely $q\ceil{(M+1)/q}$, costs $O(\log^2 N)$ per prime (assuming naive arithmetic), so $O(\sqrt N \log^2 N)$ altogether. The number of such multiples is
\begin{multline*}
 d \leq \sum_{q \leq \sqrt N} \ceil{(N - M)/q} \leq \sum_{q \leq \sqrt N} (N - M)/q + 1 \\
  = O((N - M) \log \log N + \sqrt N) = O((N - M + \sqrt N) \log \log N).
\end{multline*}
The proof is concluded in the same way as Proposition \ref{prop:enumerate}.
\end{proof}

Having dealt with these preliminaries, we now turn to computing Wilson quotients. First we give a simple algorithm that proves Theorem \ref{thm:main}, and which will serve as a template for the more involved algorithm needed for the proof of Theorem \ref{thm:main-space}. The structure of the computation bears some similarity to the parallel prefix tree in \cite{BK-adders}.
\begin{proof}[Proof of Theorem \ref{thm:main}]
First use Proposition \ref{prop:enumerate} to enumerate the primes $p \leq N$ in time $N \log^{2+\eps} N$.

Let $d = \ceil{\log_2 N}$. For each $0 \leq i \leq d$ and $0 \leq j < 2^i$ let
 \[ U_{i,j} = \left\{k \in \ZZ: j \frac{N}{2^i} < k \leq (j + 1) \frac{N}{2^i} \right\}. \]
Thus $U_{i,0}, \ldots, U_{i,2^i - 1}$ partition the interval $0 < k \leq N$ into $2^i$ sets of roughly equal size. For $0 \leq i < d$ we have the disjoint union $U_{i,j} = U_{i+1,2j} \cup U_{i+1, 2j+1}$, and $|U_{d,j}| = 0$ or $1$ for every $j$.

For each $i$, $j$ let
\[ A_{i,j} = \prod_{k \in U_{i,j}} k, \qquad\qquad S_{i,j} = \prod_{\substack{p \in U_{i,j} \\ \text{$p$ prime}}} p^2. \]
Note that $A_{i,j} = A_{i+1,2j} A_{i+1,2j+1}$, and that $A_{i,j}$ has $O(2^{-i} N \log N)$ bits. We have $A_{d,j} = 1$ or $k$ according to whether $U_{d,j} = \emptyset$ or $\{k\}$. We may compute all the $A_{i,j}$ using a product tree \cite{Ber-fastmult}, working from the bottom of the tree ($i = d$) to the top ($i = 0$). The cost at each level of the tree is $2^i (2^{-i} N \log N) \log^{1+\eps} N = N \log^{2+\eps} N$, so the total cost to compute all the $A_{i,j}$ is $N \log^{3+\eps} N$. Similarly we may compute all the $S_{i,j}$ using a product tree and the precomputed table of primes, in time $N \log^{3+\eps} N$. (In fact, because of the estimate $\sum_{p \leq N} \log p = O(N)$, this product tree takes time only $N \log^{2+\eps} N$, but we will not use this here.)

Now let
 \[ W_{i,j} = \prod_{0 \leq r < j} A_{i, r} \pmod{S_{i,j}} = \left(\left\lfloor j \frac{N}{2^i} \right\rfloor\right)! \pmod{S_{i,j}}. \]
We may compute all the $W_{i,j}$ in time $N \log^{3+\eps} N$ by working from the top of the tree to the bottom, starting with $W_{0,0} = 1$ and then using the relations
\begin{align}
 W_{i+1,2j} & = W_{i,j} \pmod{S_{i+1,2j}}, \label{eq:W-rel1} \\
 W_{i+1,2j+1} & = W_{i,j} A_{i+1,2j} \pmod{S_{i+1,2j+1}}. \label{eq:W-rel2}
\end{align}

Finally we may read the Wilson quotients off the bottom layer of the $W_{i,j}$ tree: for each $p \leq N$, let $j = \ceil{2^d p/N} - 1$. Then $U_{d,j} = \{p\}$, so $S_{d,j} = p^2$ and $W_{d,j} = (p-1)! \pmod{p^2}$.
\end{proof}

Now we consider Theorem \ref{thm:main-space}. The first step (Stage 1) is to evaluate $M! \pmod S$ where $S = \prod_{M < p \leq N} p^2$. Using a full product tree for $M!$ would lead to time complexity $M \log^{3+\eps} M$, since $\log M! = \Theta(M \log M)$. In the next proposition, we reduce this to $M \log^{2+\eps} M$ by using a space-optimised variant of the factorial algorithm of \cite{SGV-fast-algorithms}. In practice Stage 1 makes a significant contribution to the total running time, so the reduction in time by a factor of $\log M$ is significant.
\begin{prop}
\label{prop:factorial}
Let $S > 0$ be an integer with at most $B$ bits. Then $N! \pmod S$ may be computed in time
 \[ N \log^{2+\eps} N \]
and space
 \[ O(B + \sqrt N \log N \log \log N). \]
\end{prop}
\begin{proof}
Let $N! = p_1^{e_1} \cdots p_r^{e_r}$ be the prime factorisation of $N!$. For each $j$ we have
\begin{equation}
\label{eq:e-bound}
 e_j = \floor{N/p_j} + \floor{N/p_j^2} + \cdots + \floor{N/p_j^{\floor{\log N / \log p_j}}} \leq \frac{N}{p_j - 1}.
\end{equation}
Let $d = \ceil{\log_2(N+1)}$, so that $N < 2^d$, and for each $1 \leq j \leq r$ let
 \[ e_j = f_{0,j} + 2f_{1,j} + \cdots + 2^{d-1} f_{d-1,j} \]
be the binary representation of $e_j$, i.e.~with $f_{i,j} = 0$ or $1$. Then
\begin{equation}
\label{eq:prod-A}
 N! = A_0 (A_1)^2 (A_2)^4 \cdots (A_{d-1})^{2^{d-1}},
\end{equation}
where
 \[ A_i = p_1^{f_{i,1}} \cdots p_r^{f_{i,r}}. \]
Observe that if $p_j - 1 > 2^{-i} N$ then $e_j < 2^i$ by \eqref{eq:e-bound}, so $f_{i,j} = 0$. Thus actually
 \[ A_i = \prod_{p_j \leq 2^{-i} N + 1} p_j^{f_{i,j}}, \]
and we have the following estimate for the size of $A_i$:
 \[ \log A_i \leq \sum_{p \leq 2^{-i} N + 1} \log p = O(2^{-i} N). \]

We will first show how to compute $A_i \pmod S$ in time
 \[ (2^{-i} N + \sqrt N) \log^{2+\eps} N \]
and space $O(B + \sqrt N \log N \log \log N)$.

Partition the interval $1 < k \leq 2^{-i} N + 1$ into subintervals, say $T_1, \ldots, T_m$, where each subinterval, except possibly the last, has length
 \[ L = \left\lfloor \max\left(\sqrt N, \frac{B}{\log N \log \log N}\right)\right\rfloor. \]
For each subinterval $T_r$, perform the following operations.

First use Proposition \ref{prop:enumerate-interval} to enumerate the primes in $T_r$. For each subinterval, this uses space $O((L + \sqrt N) \log N \log \log N) = O(B + \sqrt N \log N \log \log N)$. The time cost for each subinterval of length $L$ is $(L + \sqrt N) \log^{2+\eps} N = L \log^{2+\eps} N$. There are at most $2^{-i} N/L$ such subintervals, so their total cost is $2^{-i} N \log^{2+\eps} N$. The last interval has length at most $2^{-i} N$, so contributes $(2^{-i} N + \sqrt N) \log^{2+\eps} N$. The time cost over all subintervals is therefore $(2^{-i} N + \sqrt N) \log^{2+\eps} N$.

Now compute $f_{i,j}$ for each $p_j \in T_r$. Using \eqref{eq:e-bound}, the time complexity is $(\log N / \log p) \log^{1+\eps} N = \log^{2+\eps} N$ for each prime, which over all subintervals is $\pi(2^{-i} N) \log^{2+\eps} N = 2^{-i} N \log^{2+\eps} N$.

Append the primes for which $f_{i,j} = 1$ to a separate buffer. Whenever the total length of that buffer reaches $B$ (i.e.~when it contains $B/\log N$ primes), or when we finish processing the last interval, use a product tree to multiply together the primes in the buffer (using space $O(B)$), and then clear the buffer to receive more primes. Accumulate the result of the product tree into a running product for $A_i \pmod S$, using a single multiplication modulo $S$ (again space usage is $O(B)$). The total time for the product trees over all intervals is $(\log A_i) \log^{2+\eps} B = 2^{-i} N \log^{2+\eps} N$, since we may certainly assume that $B = O(\log N!) = O(N \log N)$. The time for the modular multiplications is $\floor{(\log A_i) / B} B \log^{1+\eps} B = 2^{-i} N \log^{1+\eps} N$. We conclude that $A_i \pmod S$ may be computed within the promised time and space bounds.

Now let
 \[ C_i = A_i (A_{i+1})^2 \cdots (A_{d-1})^{2^{d-1-i}} \]
for $0 \leq i \leq d - 1$. We have $C_{d-1} = A_{d-1} \pmod S$ and $C_i = A_i (C_{i+1})^2 \pmod S$ for $0 \leq i \leq d - 2$. Using these relations, we compute in sequence $A_{d-1}, C_{d-1}, A_{d-2}, C_{d-2}, \ldots, A_0, C_0 \pmod S$. By \eqref{eq:prod-A}, at the end we have obtained $C_0 = N! \pmod S$. To estimate the time complexity, note that
\begin{align*}
 \log C_i & = O(2^{-i} N + 2 (2^{-i-1} N) + \cdots + 2^{d-1-i} (2^{-d+1} N)) \\ 
          & = O(2^{-i} N \log N).
\end{align*}
Therefore computing $C_i = A_i (C_{i+1})^2 \pmod S$ from $A_i \pmod S$ and $C_{i+1} \pmod S$ has time complexity $2^{-i} N \log^{2+\eps} N$. (Here we have used the fact that if $X$ and $Y$ are integers with at most $M$ bits, then $XY \pmod S$ can be computed from $X \pmod S$ and $Y \pmod S$ in time $M \log^{1+\eps} M$. Indeed if $XY < S$ then no modular reduction is performed, whereas if $XY \geq S$, we need to perform one modular reduction whose time cost is bounded by a constant multiple of the cost of the full multiplication.) The space complexity is $O(B)$, with the previous values of $A_i$ and $C_i$ discarded as we proceed. Summing over $i$, the total time cost is $N \log^{2+\eps} N$.
\end{proof}

Finally we may prove Theorem \ref{thm:main-space}.
\begin{proof}[Proof of Theorem \ref{thm:main-space}]
We must first enumerate the primes $M < p \leq N$. Using Proposition \ref{prop:enumerate-interval} directly for this would use too much space, but we may instead apply it to successive subintervals of length $K = \floor{L / \log N \log \log N}$, where $L = N - M$. The space used is $O((K + \sqrt N) \log N \log \log N) = O(L + \sqrt N \log N \log \log N) = O(L)$, plus the space needed to store the primes, namely $O((\pi(N) - \pi(M))\log N)$. To estimate the latter, note that according to \cite[Thm.~6.6]{IK-analytic} we have $\pi(N) - \pi(M) = O(L / \log L)$. Our assumption $L > \sqrt N \log N \log \log N$ then implies that $(\pi(N) - \pi(M))\log N = O(L)$. Thus the space usage is indeed $O(L)$. The time over all subintervals is $L \log^{2+\eps} N + (L/K) \sqrt N \log^{2+\eps} N = L \log^{2+\eps} N + \sqrt N \log^{3+\eps} N$.

Multiply the squares of the primes together using a product tree to obtain $S = S_{0,0} = \prod_{M < p \leq N} p^2$. The number of bits in $S$ is $B = O(L)$, so this takes space $O(L)$ and time $L \log^{2+\eps} N$.

Use Proposition \ref{prop:factorial} to compute $M! \pmod S$ in time $M \log^{2+\eps} M$ and space $O(L + \sqrt N \log N \log \log N) = O(L)$. This is Stage 1.

For Stage 2, we use a similar strategy as in the proof of Theorem \ref{thm:main}, but taking additional care to economise on space usage. Let $d = \ceil{\log_2 L}$. For each $0 \leq i \leq d$ and $0 \leq j < 2^i$ let
 \[ U_{i,j} = \left\{k \in \ZZ: M + j \frac{L}{2^i} < k \leq M + (j + 1) \frac{L}{2^i} \right\}. \]
For each $i$ this yields a partition of the interval $M < k \leq N$ into $2^i$ sets. As in Theorem \ref{thm:main}, put
\[ A_{i,j} = \prod_{k \in U_{i,j}} k, \qquad\qquad S_{i,j} = \prod_{\substack{p \in U_{i,j} \\ \text{$p$ prime}}} p^2. \]
The definition of $W_{i,j}$ is slightly different; we take
 \[ W_{i,j} = M! \prod_{0 \leq r < j} A_{i, r} \pmod{S_{i,j}} = \left(\left\lfloor M + j \frac{L}{2^i} \right\rfloor\right)! \pmod{S_{i,j}}. \]

We do not have enough space to store all of the $A_{i,j}$ and $S_{i,j}$, so we must proceed differently to the proof of Theorem \ref{thm:main}. We will use a strategy similar to the proof of 
\cite[Lemma 2.1]{vzGS-frobenius}.

We begin at the top of the tree with $W_{0,0} = M! \pmod{S_{0,0}}$, which was computed above using Proposition \ref{prop:factorial}. As in the proof of Theorem \ref{thm:main}, we use relations \eqref{eq:W-rel1} and \eqref{eq:W-rel2} to work our way down the tree. Every new pair of values $W_{i+1,2j}$ and $W_{i+1,2j+1}$ overwrites the previous value of $W_{i,j}$. For fixed $i$, the total size of the $W_{i,j}$ at level $i$ is $O(L)$, so the space for storing the $W_{i,j}$ never exceeds $O(L)$.

For the top $\ell = \floor{2 \log_2 \log N}$ levels of the tree, we recompute each required $A_{i,j}$ and $S_{i,j}$ as we encounter them, discarding intermediate values (i.e.~$A_{i,j}$ and $S_{i,j}$ from lower levels of the product tree) as we proceed. Also, in the evaluation of \eqref{eq:W-rel2}, we do not compute $A_{i+1,2j+1}$ exactly, but rather only modulo $S_{i+1,2j+1}$, by reducing as appropriate during the product tree computation. The time complexity contributed by each level of the tree is thus $L \log^{3+\eps} N$ (this is a factor of $\log N$ more than in Theorem \ref{thm:main}, due to the recomputations), but over the first $\ell$ levels this amounts to only $L \log^{3+\eps} N \log \log N = L \log^{3+\eps} N$.

When we reach level $\ell$, we switch back to the strategy of Theorem \ref{thm:main}. For each $j$ at level $\ell$, we compute the entire trees beneath $A_{\ell,j}$ and $S_{\ell,j}$. This requires space $O(\log(A_{\ell,j}) \log N) = O(2^{-\ell} L \log^2 N) = O(L)$. The time contribution from each level is $L \log^{2+\eps} N$, so over all levels is $L \log^{3+\eps} N$. The Wilson quotients are extracted from the $W_{d,j}$ just as in Theorem \ref{thm:main}.
\end{proof}

\section{Factorial identities modulo $p^2$}
\label{sec:identities}

Let $e$ be an even divisor of $p - 1$, and let $f = (p-1)/e$. In this section we describe a method for reducing computation of $(p-1)! \pmod{p^2}$ to that of $f! \pmod{p^2}$.

As mentioned in the introduction, identity \eqref{eq:reduce-2}, corresponding to the case $e = 2$, has been applied to the computation of Wilson quotients for almost a century. The cases $e = 4$ and $e = 6$ were introduced by \cite{CDP-search}.

Our method can be applied in principle to any $e$. The simplest case, and the only case we will describe in this paper, is when the $e$-th cyclotomic field over $\QQ$ has class number 1. It is known that this occurs for precisely the following values of $e$ (\cite[Ch.~11]{Was-cyclotomic}):
\begin{multline*}
2, 4, 6, 8, 10, 12, 14, 16, 18, 20, 22, 24, 26, 28, 30, \\
 32, 34, 36, 38, 40, 42, 44, 48, 50, 54, 60, 66, 70, 84, 90,
\end{multline*}
and these are the values of $e$ that we used in our implementation.

It is straightforward to modify the algorithms given in the proof of Theorem \ref{thm:main-space} to compute $f! \pmod{p^2}$ instead of $(p-1)! \pmod{p^2}$. For example, given a set $T$ of primes $p$ lying in the interval $M < p \leq N$ and satisfying $p = 1 \pmod e$, the modified Stage 1 involves using Proposition \ref{prop:factorial} to compute $\floor{M/e}! \pmod{\prod_{p \in T} p^2}$.

To apply this to the main Wilson prime search, each prime $p$ is assigned to the `best' possible $e$ for that prime, i.e.~the largest divisor of $p - 1$ appearing in the above list. Then for each $e$, our strategy is to use the (suitably modified) algorithm of Theorem \ref{thm:main-space} to compute $w_p$ for all $p$ assigned to $e$.

It is a difficult theoretical problem to analyse the savings that accrue from this strategy. If we assume that the amount of RAM is fixed, then Stage 1 will dominate for sufficiently large $N$. In Stage 1 we expect a speedup by roughly a linear factor of $e$, since we are only computing $\floor{M/e}!$ rather than $M!$. Therefore, in the limit of large $N$, we expect a savings of a factor of $e$ for the primes assigned to $e$.

In practice however these ideal conditions are not met. Stage 2 does make a significant contribution, especially for larger values of $e$. The effect of $e$ on Stage 2 is complex. As $e$ increases, a fixed interval $M < p \leq N$ will contain fewer and fewer primes of interest. The number of such primes depends in a complicated way on the complete list of admissible $e$. To make best use of available RAM, for larger $e$ we will generally choose a larger interval, so that the number of primes in the interval is roughly constant, but the relationship is not linear.

In addition, we must take into account the cost of deducing $(p-1)! \pmod{p^2}$ from $((p-1)/e)! \pmod{p^2}$. We refer to this step of the computation as Stage 3. We have not attempted to give a theoretical bound for the cost of Stage 3. In general it becomes more expensive as $e$ increases. In our computation it accounted for only a few percent of the total running time (see Table \ref{tab:stages}).

Let us estimate the overall savings, over many primes, under the assumption that the speedup is linear in $e$, and ignoring the cost of Stage 3. Let $S$ be a set of permissible values of $e$, for example, the set $\{2, 4, \ldots, 84, 90\}$ given above. We assume that for each $e \in S$, we apply the above strategy to those primes $p$ for which $e$ is the largest divisor of $p - 1$ that appears in $S$. Let $Q_S = \LCM(S)$. For $k \in (\ZZ/Q_S\ZZ)^*$, let $b_S(k) = \max\{e \in S: k = 1 \pmod e \}$. Then the expected savings is
 \[ R_S = \frac{1}{\phi(Q_S)} \sum_{k \in (\ZZ/Q_S\ZZ)^*} \frac{1}{b_S(k)}. \]

For example, if we only use identity \eqref{eq:reduce-2}, then $S = \{2\}$, $Q_S = 2$, and $R_S = 1/2$, so we save a factor of $2$ over the naive algorithm.

The identities used in \cite{CDP-search} correspond to choosing $S = \{2, 4, 6\}$, in which case $Q_S = 12$ and $R_S = (1/6 + 1/4 + 1/6 + 1/2)/4 = 13/48$, saving a further factor of $24/13 \approx 1.85$.

Taking $S$ to be the full set $S = \{2, 4, \ldots, 84, 90\}$, we have
 \[ Q_S = 6983776800 = 2^5 \cdot 3^3 \cdot 5^2 \cdot 7 \cdot 11 \cdot 13 \cdot 17 \cdot 19. \]
A brute force computation finds that
 \[ R_S = \frac{22695187978681}{201921527808000} \approx 0.112, \]
indicating a further savings of a factor of roughly 2.41 compared to \cite{CDP-search}.

Now we explain the reduction. Fix a primitive $e$-th root of unity $\omega \in \ZZ_p$. Let $\Gamma_p : \ZZ_p \to \ZZ_p^*$ denote the $p$-adic gamma function. The next proposition, whose proof is adapted from \cite[Thm.~9.3.1]{BEW-gauss-jacobi}, gives a congruence between $(p-1)!/f!^e$ and a special value of the $p$-adic gamma function.
\begin{prop}
\label{prop:ratio}
Let
 \[ C = \frac1p \sum_{j=1}^{e-1} \big( (1 - \omega^j)^p - (1 - \omega^j) \big) \in \ZZ_p. \]
Then
 \[ \frac{(p-1)!}{f!^e} = -\Gamma_p(1/e)^e (1 + pC) \pmod{p^2}. \]
\end{prop}
\begin{proof}
Let $M = p^2 - (p^2 - 1)/e = p^2 - f(p+1)$. Then $M = 1/e \pmod{p^2}$ and $1 \leq M < p^2$. By the definition and elementary properties of $\Gamma_p(x)$ (see for example \cite[Ch.~14]{Lan-cyclotomic-combined}) we have
\begin{align*}
 \Gamma_p(1/e) & = \Gamma_p(M) \pmod{p^2} \\
	          & = - \prod_{\substack{1 \leq j < M \\ p \ndivides j}} j \pmod{p^2}.
\end{align*}
Splitting the product into blocks of length $p$ we obtain
\[ \Gamma_p(1/e) = - \left(\prod_{k=0}^{\ceil{M/p} - 1} \prod_{r=1}^{p-1} (kp + r) \right) \left( \prod_{j=M}^{\ceil{M/p}p - 1} j \right)^{-1} \pmod{p^2}. \]
Since $\ceil{M/p} = p - f + \floor{f/p} = p - f$,
\[ \Gamma_p(1/e) = - \left(\prod_{k=0}^{p - f - 1} \prod_{r=1}^{p-1} (kp + r) \right) \left( \prod_{j=p^2 - fp - f}^{p^2 - fp - 1} j \right)^{-1} \pmod{p^2}. \]
For the first term, observe that for any $k \in \ZZ$ we have
 \[ \frac{\prod_{r=1}^{p-1} (kp + r)}{(p-1)!} = \prod_{r=1}^{p-1} (1 + kp/r) = 1 + kp \sum_{r=1}^{p-1} 1/r = 1 \pmod{p^2}. \]
Therefore
 \[ \prod_{k=0}^{p - f - 1} \prod_{r=1}^{p-1} (kp + r) = (p-1)!^{p - f} \pmod{p^2}. \]
For the second term,
 \[ \prod_{j=p^2 - fp - f}^{p^2 - fp - 1} j = \prod_{j=-fp - f}^{- fp - 1} j = (-1)^f \prod_{r=1}^f (r + fp) \pmod{p^2}. \]
To evaluate this last product, note that
 \[ \frac{\prod_{r=1}^f (r + fp)}{f!} = \prod_{r=1}^f (1 + fp/r) = 1 + fp \sum_{r=1}^f 1/r \pmod{p^2}. \]
Moreover, for any $1 \leq j \leq e - 1$,
\begin{align*}
 \frac{(1 - \omega^j)^p - (1 - \omega^j)}p
   & = \frac1p \sum_{k=1}^{p-1} \binom{p}{k} (-\omega^j)^k \\
   & = \sum_{k=1}^{p-1} \frac{(p-1)(p-2) \cdots (p - k + 1)}{k (k-1) \cdots 1} (-\omega^j)^k \\
   & = -\sum_{k=1}^{p-1} \frac{(k-1)!}{k!} \omega^{jk} = -\sum_{k=1}^{p-1} \omega^{jk}/k \pmod{p}.
\end{align*}
Thus
 \[ C = -\sum_{j=1}^{e-1} \sum_{k=1}^{p-1} \omega^{jk}/k = -\sum_{k=1}^{p-1} \frac1k \sum_{j=1}^{e-1} \omega^{jk}  \pmod{p}. \]
Since
 \[ \sum_{j=1}^{e-1} (\omega^k)^j = -1 + \begin{cases} e & \text{if $e \divides k$}, \\ 0 & \text{otherwise}, \end{cases} \]
we get
 \[ C = - \sum_{k=1}^{p-1} \frac{1}{k} - e \sum_{r=1}^f \frac{1}{er} = - \sum_{r=1}^f  1/r \pmod{p}. \]

Putting everything together, we have
 \[ \Gamma_p(1/e) = \frac{-(p-1)!^{p-f} (-1)^f}{f! (1 - fpC)} \pmod{p^2}. \]
From Wilson's theorem we have $(p-1)!^p = -1 \pmod{p^2}$, and so
 \[ \Gamma_p(1/e)^e = \frac{(p-1)!^{-ef}}{f!^e (1 - fpC)^e} = \frac{(p-1)!^{-p+1}} {f!^e (1 - efpC)} = \frac{-(p-1)!}{f!^e (1 + pC)} \pmod{p^2} \]
Rearranging, we obtain the desired formula.
\end{proof}

Next we will use the Gross--Koblitz formula to relate $\Gamma_p(1/e)$ to a certain Gauss sum. Let $K = \QQ(\zeta_e)$, where $\zeta_e$ is a primitive $e$-th root of unity. The ring of integers of $K$ is $O_K = \ZZ[\zeta_e]$. Let $\omega_0 \in \ZZ$ be an integer congruent to $\omega \pmod p$, and let $P = (p, \zeta - \omega_0)$. Then $P$ is a prime ideal of $O_K$ of degree $1$ lying above $p$, i.e.~$O_K/P \cong \FF_p$. Let $\chi : \FF_p^* \to K^*$ be the $(-f)$-th power of the Teichm\"uller character; that is, $\chi(u) = u^{-f} \pmod P$ for any $u \in \FF_p^*$. Define the Gauss sum
 \[ S(\chi) = \sum_{j=1}^{p-1} \chi(j) \zeta_p^j \in K(\zeta_p), \]
where $\zeta_p$ is a primitive $p$-th root of unity.
\begin{prop}
\label{prop:gross-koblitz}
We have $S(\chi)^e \in K$. Regarding $K$ as embedded in $\QQ_p$ via the map that sends $\zeta_e$ to $\omega$, we have
\[  -\Gamma_p(1/e)^e = \frac{(-S(\chi))^e}{p}. \]
\end{prop}
\begin{proof}
The first statement follows from \cite[Ch.~1, Thm.~1.3(i)]{Lan-cyclotomic-combined}. The second statement is a consequence of the Gross--Koblitz formula, for example \cite[Ch.~15, Thm.~4.3]{Lan-cyclotomic-combined}. In the notation of \cite{Lan-cyclotomic-combined}, take $r = 1$, $q = p$, $a = p - 1 - f$. The above formula falls out after taking $e$-th powers.
\end{proof}

The final ingredient is the Stickelberger factorisation of the ideal of $K$ generated by $S(\chi)^e$. For $c \in (\ZZ/e\ZZ)^*$, let $\sigma_c$ denote the automorphism of $K/\QQ$ that sends $\zeta_e$ to $\zeta_e^c$.
\begin{prop}
\label{prop:stickelberger}
 \[ (S(\chi)^e) = \prod_{\substack{c=1 \\ (c, e) = 1}}^{e-1} \sigma_{c^{-1}}(P)^c. \]
\end{prop}
\begin{proof}
Raise both sides of \cite[Ch.~1, Thm.~2.2]{Lan-cyclotomic-combined} to the power of $e$.
\end{proof}
\begin{prop}
\label{prop:root-of-unity}
Suppose that $P$ is principal, and let $\theta$ be a generator. Let
 \[ \beta = \prod_{\substack{c=1 \\ (c,e) = 1}}^{e-1} \sigma_{c^{-1}}(\theta)^c \in O_K. \]
Then
 \[ S(\chi)^e = \zeta_e^i \beta \]
for some $0 \leq i < e$.
\end{prop}
\begin{proof}
By Proposition \ref{prop:stickelberger}, $S(\chi)^e$ and $\beta$ differ by a unit of $O_K$. Moreover,
 \[ \sigma_{-1}(\beta) = \prod_c \sigma_{-c^{-1}}(\theta)^c = \prod_c \sigma_{c^{-1}}(\theta)^{e-c} \]
so
 \[ \beta \sigma_{-1}(\beta) = \prod_c \sigma_{c^{-1}}(\theta)^e = N_{K/\QQ}(\theta)^e = N(P)^e = p^e. \]
Thus the image of $\beta$ under every complex embedding $K \to \CC$ has absolute value $p^{e/2}$. But $S(\chi)^e$ has the same property \cite[p.~4]{Lan-cyclotomic-combined}. Therefore $S(\chi)^e/\beta$ has absolute value $1$ in every complex embedding, and so is a root of unity in $K$ \cite[Lemma 1.6]{Was-cyclotomic}. Since $e$ is even, every root of unity is a power of $\zeta_e$, and the conclusion follows.
\end{proof}

\begin{thm}
\label{thm:cyclotomic}
Let $p = 1 \pmod e$, where $e$ is even. Assume that $K = \QQ(\zeta_e)$ has class number $1$. Assume we are given as input:
\begin{itemize}
\item a primitive $e$-th root of unity in $\FF_p^*$, represented as an integer $1 \leq \omega_0 < p$,
\item a generator $\theta$ of the ideal $P = (p, \zeta_e - \omega_0)$, represented as $\theta = g(\zeta_e)$ for some polynomial $g \in \ZZ[x]$ of degree less than $\phi(e)$, and
\item $f! \pmod{p^2}$.
\end{itemize}
Then we may compute $(p-1)! \pmod{p^2}$ using $O(e^2 + e\log p)$ arithmetic operations on integers with $O(\log p)$ bits.
\end{thm}
The big-$O$ estimates given in the above theorem are strictly speaking meaningless, since they only apply to finitely many $e$. We give the estimates anyway as an indication of how the running time might reasonably be expected to behave in practice.

We may compute a suitable $\omega_0$ using a simple probabilistic algorithm as follows. Select a random $1 \leq x \leq p - 1$. Then $\omega_0 = x^f \pmod p$ has order exactly $e$ with probability $\phi(e)/e \geq 1/e$. We can compute the exact order using at most $e$ arithmetic operations in $\ZZ/p\ZZ$. This is repeated until we find a suitable $\omega_0$.

The computational bottleneck is actually in finding $\theta$, which we consider after the proof of the theorem.

\begin{proof}
We will take `arithmetic operation' to mean an addition or multiplication modulo $p$, $p^2$ or $p^3$.

We first lift the root of unity, putting $\omega_1 = \omega_0^{p^2} \pmod{p^3}$, so that $\omega_1 = \omega \pmod{p^3}$. This requires $O(\log p)$ arithmetic operations. We next compute the powers $\omega_1^i$ for $0 \leq i < e$, using $O(e)$ arithmetic operations. Computing $C$ from Proposition \ref{prop:ratio} requires $O(e \log p)$ arithmetic operations.

Let $\gamma$ be the image in $\ZZ_p$ of $\beta/p$, where $\beta$ is as in Proposition \ref{prop:root-of-unity}, i.e.~
 \[ \gamma = \frac1p \prod_{\substack{c=1 \\ (c,e) = 1}}^{e-1} g(\omega^{c^{-1}})^c. \]
With this formula, we may compute $\gamma \pmod{p^2}$ using $O(e^2)$ arithmetic operations. Combining Propositions \ref{prop:ratio}, \ref{prop:gross-koblitz}, \ref{prop:stickelberger} and \ref{prop:root-of-unity}, we have
 \[ \omega^{-i} (p-1)! = (-f!)^e \gamma (1 + pC) \pmod{p^2} \]
for some $0 \leq i < e$, so we can compute $\omega^{-i} (p-1)! \pmod{p^2}$ using a further $O(\log e)$ operations. However, we know that $(p-1)! = -1 \pmod p$, so we can determine $i$ by comparing with the tabulated powers of $\omega$.
\end{proof}

Before discussing the computation of $\theta$, we illustrate Theorem \ref{thm:cyclotomic} with a numerical example. Take $p = 3333331$, $e = 18$, $f = 185185$, and the $18$th root $\omega_0 = 1819843$. The Teichm\"uller lift is
 \[ \omega = 1819843 + 1422487p + 90367p^2 \pmod{p^3}, \]
and
 \[ C = \frac{(1 - \omega)^p - (1 - \omega) + \cdots + (1 - \omega^{17})^p - (1 - \omega^{17})}p = 418399 \pmod p. \]
Using the cyclotomic GCD algorithm discussed below, we find a generator $\theta = g(\zeta_e)$ of $P = (p, \zeta_e - \omega_0)$ given by
 \[ g(x) = -5x^5 - 10x^4 + 7x^3 + 3x^2 + 10x - 4. \]
Then
\begin{align*}
 \gamma & = \frac1p g(\omega) g(\omega^{11})^5 g(\omega^{13})^7 g(\omega^5)^{11} g(\omega^7)^{13} g(\omega^{17})^{17} \\
 & = 1628187 + 503367p \pmod{p^2}.
\end{align*}
Now assuming that we have computed
 \[ f! = 461190 + 275007p \pmod{p^2}, \]
we find that
 \[ \omega^{-i} (p-1)! = (-f!)^e \gamma (1 + pC) = 1780730 + 2171988p \pmod{p^2}. \]
Comparing with the powers of $\omega$, we find that $\omega_0^{3} = -1780730 \pmod p$, so $i = 3$ and
 \[ (p-1)! = 3333330 + 27003p \pmod{p^2}. \]
We conclude that $w_p = 27004$.

Now we consider the problem of computing $\theta$. The standard approach to the ideal generator problem is based on lattice reduction (see for example \cite{Coh-compnt}), and indeed there exist highly optimised implementations in software packages such as Pari/GP \cite{PARI-2.3.5}.

After some experimentation we settled on a different approach, which we found to be considerably faster than Pari in practice. Our algorithm is closer in spirit to the elementary Euclidean GCD algorithm. We emphasise that this is not a general-purpose algorithm for finding ideal generators: it assumes that $K$ has class number $1$, and also uses the fact that we know in advance that the generator is an irreducible element whose norm is not too small. In addition we are unable to prove that the `algorithm' terminates. In practice we find that it does terminate quite quickly. Pseudocode is shown in Algorithm \ref{algo:gcd} below. The algorithm is applied to the inputs $X = p$ and $Y = \zeta_e - \omega_0$, and their GCD is precisely the desired $\theta$.

\begin{algorithm}
\label{algo:gcd}
\SetAlgoLined
\DontPrintSemicolon
\KwIn{$X, Y \in O_K$ \newline $S = $ precomputed list of elements of $O_K$ of small norm}
\KwOut{A generator of $(X, Y)$}
\medskip
\While{$X \neq 0$ and $Y \neq 0$}{
  \lIf{$N(X) < N(Y)$}{swap $X$ and $Y$} \nllabel{line:swap}\;
  $Q \assign $ an element of $O_K$ near $X/Y$ \nllabel{line:quotient}\;
  $Z \assign X - QY$ \nllabel{line:reduce}\;
  \lIf{$N(Z) < N(Y)$}{$X \assign Z$} \nllabel{line:update}\;
  \Else{
    $U \assign$ randomly selected element of $S$ \nllabel{line:choose-U}\;
    \lIf{$U \divides Y$}{$Y \assign Y / U$} \lElse{$X \assign XU$} \nllabel{line:update-2}\;
  }
}
\lIf{$X = 0$}{\KwRet{$Y$}} \lElse{\KwRet{$X$}}\;
\caption{Heuristic cyclotomic GCD}
\end{algorithm}

Several aspects of the algorithm deserve further discussion.

All elements of $O_K$ appearing in the algorithm are represented exactly, as $\ZZ$-linear combinations of the basis elements $\{1, \zeta_e, \ldots, \zeta_e^{d-1}\}$, where $d = \phi(e) = [K:\QQ]$, i.e.~as polynomials in $\zeta_e$. We first attempt to run the algorithm with all coefficients represented by signed 64-bit integers, and ignoring all overflows. If the algorithm terminates, we can check the output by verifying that the proposed $\theta$ divides both $p$ and $\zeta_e - \omega_0$. This usually succeeds. If it is incorrect, or if the algorithm runs for too long without terminating, we restart it. If this fails several times, we switch to an implementation that uses an arbitrary precision representation for the coefficients. This eliminates the possibility of overflow, so that if the algorithm terminates, the output is guaranteed to be correct. Again, if it runs for too long, we restart it. In practice this always eventually succeeds.

Exact multiplication of elements of $O_K$ (lines \ref{line:reduce} and \ref{line:update-2}) is achieved by naive polynomial multiplication followed by reduction modulo the cyclotomic polynomial $\phi_e(x)$. Exact division (line \ref{line:update-2}) is achieved by the formula $X/Y = X \prod_{\sigma \neq 1} \sigma(Y) / N(Y)$, where the denominator $N(Y) = \prod_{\sigma} \sigma(Y)$ is a rational integer. Here $\sigma$ denotes an automorphism of $K$, which is evaluated by cyclic permutation of coordinates followed by reduction modulo $\phi_e(x)$.

Let $\tau_1, \overline{\tau_1}, \ldots, \tau_{d/2}, \overline{\tau_{d/2}}$ be the complex embeddings $K \hookrightarrow \CC$, and let $\tau = (\tau_1, \ldots, \tau_{d/2}) : K \to \CC^{d/2}$ be the corresponding vector of embeddings. For each variable $V$ in Algorithm \ref{algo:gcd}, we also maintain a second representation, namely a double-precision floating point approximation to $\tau(V)$.

In lines \ref{line:swap} and \ref{line:update}, the norms are approximated by multiplying together the coordinates of $\tau(V)$, rather than by computing an exact norm in $\ZZ$.

In line \ref{line:quotient}, we first approximate $\tau(X/Y)$ by computing $\tau_i(X) / \tau_i(Y)$ (as a floating-point complex number) for each $i$. Applying the inverse of $\tau$ yields an approximation to $X/Y$ in $K \otimes_\QQ \RR$. We select $Q$ by simply rounding each coordinate to the nearest integer. In the ideal situation we will have $N(X/Y - Q) < 1$. If this holds, then line \ref{line:update} will succeed in updating $X$, and then we have made some progress in reducing the norm. However there is no guarantee that $N(X/Y - Q) < 1$ will occur. One possibility is that there exists some $Q' \in O_K$ such that $N(X/Y - Q') < 1$, but that our simple-minded method for selecting $Q$ did not locate it. To mitigate against this, we make a few attempts to adjust the coordinates of $Q$ to locate a suitable $Q'$. This may still fail, and moreover it may turn out that there does not exist \emph{any} $Q'$ with the right property. This may occur if $K$ is not Euclidean with respect to the norm; for example it is known that $\QQ(\zeta_{32})$ has this property \cite{Len-euclidean}. In this case, we will fall through to lines \ref{line:choose-U}--\ref{line:update-2}.

The goal of lines \ref{line:choose-U}--\ref{line:update-2} is to make some random perturbation, in the hope that we will be lucky in finding a good $Q$ on the next iteration. In our implementation, we take $S$ to be the set of elements of $O_K$ of norm $q$, where $q$ is the smallest prime $q = 1 \pmod e$ (i.e.~take all the conjugates of a generator of any prime ideal dividing $qO_K$). If we are lucky enough that $U$ divides $Y$, then we know $U$ cannot divide $X$, since we have assumed that the GCD has norm $p$, which is much larger than $q$. Thus dividing $Y$ by $U$ does not change the GCD. Otherwise, we simply multiply $X$ by $U$ and continue. This cannot change the GCD for the same reason.

The rationale for this perturbation strategy is as follows. If $X/Y$ is sufficiently close to an integer, then our method for selecting $Q$ should find it. Otherwise, $UX/Y$ is likely to be `randomly distributed' modulo the integer lattice, and there is a reasonable chance that it will be close to an integer. We have not attempted to formulate this argument precisely or prove anything about it.

Finally we discuss the issue of units. Whenever we compute a new element of $O_K$, say $X$, we examine the size of its coefficients, and compare this to $N(X)$. If the coefficients are too large, we apply a balancing procedure, replacing $X$ by $u^{-1} X$ for a suitable unit $u \in O_K^*$. This of course does not alter the GCD. An extreme example of an `unbalanced' element is a high power of a nontrivial unit $u \in O_K^*$, which has large coefficients but norm $1$. Without this balancing step, we soon encounter coefficient explosion (and overflow).

The condition we used to test for unbalancedness in our implementation is as follows: if $X = c_0 + c_1 \zeta_e + \cdots + c_{d-1} \zeta_e^{d-1}$, we declare that $X$ is unbalanced if $\frac1d\sum_{i=0}^{d-1} |c_i| > 10 |N(X)|^{1/d}$. There is no particular theoretical justification for this particular measure of size, nor of the choice of constant $10$. We used it because it is fast to evaluate and seems to give good results in practice. 

To balance an element $X$ we proceed as follows. (This strategy is inspired by the definition of `unbalanced' in \cite{Wik-lary}.) Consider the logarithmic embedding $L: O_K\setminus\{0\} \to \RR^{d/2}$ defined by $a \mapsto (\log|\tau_i(a)|)_i$. By Dirichlet's unit theorem, the image of the unit group $O_K^*$ under this map is a lattice of full rank in the hyperplane $t_0 + \cdots + t_{d/2-1} = 0$. The vector $(\log|\tau_i(X)| - \frac1d\log|N(X)|)_i$ lies in this hyperplane. Armed with a precomputed list of generators of $O_K^*$ (obtained for example via Pari), we may therefore use simple linear algebra over $\RR$ to select a unit $u$ so that $\log|\tau_i(u)|$ is close to $\log|\tau_i(X)| - \frac1d\log|N(X)|$ for all $i$. Then we replace $X$ by $u^{-1} X$ and continue. The rationale is that our choice of $u$ ensures that $|\tau_i(u^{-1} X)|$ is close to $|N(X)|^{1/d}$ for all $i$, so that the coefficients of $u^{-1} X$ will be reasonably small (although they might not actually satisfy the test for balancedness mentioned in the previous paragraph).

\section{Implementation and hardware}
\label{sec:implementation}

Our implementation is written in C, using OpenMP for parallelisation at the level of the individual compute node. We used the GMP library \cite{gmp-5.0.5} for multiple-precision integer arithmetic, with the following important exception.

For very large integer multiplications --- for operands exceeding around $10^7$ bits, depending on the hardware --- we switch to our own implementation based on number-theoretic transforms (NTTs). This proceeds by splitting the input into small chunks of perhaps several words each, converting the problem to that of multiplying polynomials in $\ZZ[x]$. This is then achieved by reducing modulo several suitable 62-bit primes $q$, multiplying the polynomials using FFTs over $\ZZ/q\ZZ$, and reconstructing the product in $\ZZ[x]$ via the Chinese Remainder Theorem. The FFT arithmetic is optimised using techniques described in \cite{Har-ntt}. To ensure the running time behaves smoothly as a function of the input size, we allow the number of primes to vary between 3 and 6, and we select a transform length of the form $2^k 3^\ell$ where $0 \leq \ell \leq 6$; that is, we use mainly radix-$2$ transforms, but allow a few layers of radix-$3$ transforms. We use a strategy similar to Bailey's trick \cite{Bai-fft} to improve memory locality.

The main reason that we did not use GMP's large integer multiplication code is that GMP does not take advantage of multiple cores in a shared memory environment. In contrast, our implementation is parallelised using OpenMP. This is crucial, because in Stage 1, the average complexity per prime is inversely proportional to the amout of RAM available. To make effective use of $n$ cores, it is not good enough to process $n$ intervals separately using one core each, since each core will have only $1/n$ of the available RAM, and will run in effect at $1/n$ of the speed. We must actually parallelise within the integer arithmetic, to get all cores working cooperatively on a single interval.

Furthermore, our integer multiplication code is optimised heavily in favour of conserving memory. Its performance varies across platforms, but in all cases is competitive with GMP. For example, on a node of Katana (see below), multiplying two 1-gigabyte integers took 178s using GMP, with peak memory usage 9.1GB. Our code performs the same multiplication in 121s using only 5.3GB; running on 8 cores it takes 20s (a 6-fold speedup), using the same memory.

A natural extension of this idea, which we did not pursue, is to increase the effective RAM available by making use of the fast networks on modern HPC systems to treat several nodes as a single computational unit. Whether this yields any speedup in searching for Wilson primes is an interesting question for future research.

We ran our implementation over a period of about four months on several clusters at New York University (``Cardiac'', ``Bowery'', and ``Union Square''), the University of New South Wales (``Katana'' and ``Tensor''), and the National Computational Infrastructure facility at the Australian National University (``Vayu''). Table \ref{tab:clusters} summarises the characteristics of the nodes on each cluster, and the total CPU time expended on each cluster. Table \ref{tab:stages} gives a breakdown of the total CPU time into Stage 1, Stage 2 and Stage 3.

In the previous section it was pointed out that Stage 1 should dominate the computation for sufficiently large $p$. The data in Table \ref{tab:stages} shows that we have not yet reached this region. A more detailed accounting shows this behaviour beginning to occur in some parts of the computation; for example, for $e = 2$, on the machines with 32GB RAM, we found that Stage 1 starts to dominate for $p$ around $5 \times 10^{12}$. The threshold increases with $e$ and with the amount of RAM per node.

\begin{table}
\begin{tabular}{llrr}
\toprule
Cluster      & Architecture     & RAM (GB) & Core-hours \\
\midrule
Cardiac      & AMD Barcelona    & 32          & 465\,000 \\
             & 16 cores, 2.3GHz \\
Bowery       & Intel Nehalem    & 48/96/256   & 263\,000 \\
             & 12 cores, 2.67--3.07GHz \\
Union Square & Intel Xeon       & 16/32       & 154\,000 \\
             & 8 cores, 2.33GHz \\
Tensor       & Intel Xeon       & 16/24       & 145\,000 \\
             & 8 cores, 3.0GHz \\
Katana       & Intel Xeon       & 24/96/144   & 123\,000 \\
             & 12 cores, 2.8--3.06GHz \\
Vayu         & Intel Nehalem    & 24          &  13\,000\\
             & 8 cores, 2.93GHz \\
\bottomrule
\end{tabular}
\caption{Cluster data}
\label{tab:clusters}
\end{table}

\begin{table}
\begin{tabular}{lr}
\toprule
 & Core-hours \\
\midrule
Stage 1 & 464\,000 \\
Stage 2 & 655\,000 \\
Stage 3 & 44\,000 \\
\bottomrule
\end{tabular}
\caption{Breakdown of CPU time}
\label{tab:stages}
\end{table}

We used a client-server strategy to distribute work among the clusters. A master script ran on a server at NYU. When a compute node is ready to begin work, it sends a request via HTTP to the server. The server is responsible for choosing a value of $e$ (as in Section \ref{sec:identities}) and a range of primes $M < p < N$ to assign to that node. This basic outline is complicated by the fact that the time needed to complete a single block was generally much longer than the running time permitted for a single job by each cluster's job scheduler. It was therefore necessary to serialise intermediate computations to disk at appropriate intervals, and reload them by another job later on. Load balancing was also complicated by varying cluster availability over the duration of the project.

Any computation of this size is bound to run into hardware failures and other problems. We took several measures to validate our results.

First, for each $p$ we check that our proposed valued for $(p-1)! \pmod{p^2}$ satisfies $(p-1)! = -1 \pmod p$. Second, in the notation of the proof of Theorem \ref{thm:cyclotomic}, we check that $(-f!)^e \gamma$ is an $e$th root of unity modulo $p$. This simultaneously provides a strong verification of the cyclotomic GCD computation and of the computation of $f!$, at least modulo $p$.

Finally, we wrote a completely independent program to compute $w_p$ using the $p^{1/2 + \eps}$ algorithm of \cite{BGS-recurrences}, together with identity \eqref{eq:reduce-2} (but none of the results of Section \ref{sec:identities}). The underlying polynomial arithmetic is handled by the NTL library \cite{ntl-5.5.2}. We ran this implementation on the $27\,039\,026$ saved residues and found complete agreement. This computation was run on Katana and Tensor, together with a Condor cluster, utilising idle time on machines in the School of Mathematics and Statistics at UNSW; it took $440\,000$ CPU hours altogether.

\section*{Acknowledgements}

We thank the HPC facilities and staff at NYU, UNSW and NCI for providing computational resources and support. Thanks to Richard Brent, Claus Diem, Felix Fr\"ohlich, Mark Rodenkirch for helpful discussions, and a referee for their comments that simplified the presentation. Several members of mersenneforum.org used their personal computers to search up to $4 \times 10^{11}$ with an early implementation written by the second author.

\bibliographystyle{amsplain}
\bibliography{wilson}

\end{document}